\documentclass[12pt,twoside, a4paper]{amsart}

\usepackage{mathptmx}
\usepackage{amssymb}
\usepackage{eucal}
\usepackage{amsxtra}
\usepackage{verbatim}
\usepackage{enumerate}
\usepackage[english]{babel}
\usepackage{bm}

\usepackage{setspace}
\usepackage[body={17cm,21cm}]{geometry}
\usepackage[T1]{fontenc}
\usepackage{graphicx}
\usepackage{mathrsfs}
\usepackage{amscd,latexsym,amsthm,amsfonts,amssymb,amsmath,amsxtra}
\usepackage[colorlinks, urlcolor=blue,  citecolor=blue]{hyperref}
\usepackage{enumerate}
\usepackage[all]{xy}%
\providecommand{\U}[1]{\protect\rule{.1in}{.1in}}
\RequirePackage{amsmath}
\RequirePackage{amssymb}
\usepackage{comment}
\theoremstyle{plain}

\newtheorem{theorem}{Theorem}

\newtheorem{lemma}[theorem]{Lemma}

\newtheorem{theorem?}{Theorem(?)} [section]
\newtheorem{proposition?}[theorem]{Proposition(?)}
\newtheorem{lemma?}[theorem]{Lemma(?)}
\newtheorem{corollary?}[theorem]{Corollary(?)}

\newtheorem*{theorem*}{Theorem}
\newtheorem*{proposition*}{Proposition}
\newtheorem*{lemma*}{Lemma}
\newtheorem*{corollary*}{Corollary}
\newtheorem*{question*}{Question}
\newtheorem*{conjecture*}{Conjecture}
\newtheorem*{claim*}{Claim}

\newtheorem*{introtheorem*}{Theorem}
\newtheorem*{introproposition*}{Proposition}
\newtheorem*{introlemma*}{Lemma}
\newtheorem*{introcorollary*}{Corollary}

\theoremstyle{definition}

\newtheorem*{definition*}{Definition}
\newtheorem*{example*}{Example}

\newtheorem{question}[theorem]{Question}

\theoremstyle{remark}
\newtheorem{remark}[theorem]{Remark}

\newtheorem*{remark*}{Remark}

\numberwithin{equation}{section}
\numberwithin{theorem}{section}

\DeclareSymbolFont{rsfs}{U}{rsfs}{m}{n}
\DeclareSymbolFontAlphabet{\mathcal}{rsfs}

\newcommand{\ZZ}{{\mathbb{Z}}}
\newcommand{\QQ}{{\mathbb{Q}}}

\newcommand{\Pic}{{\rm Pic}}

\newcommand{\Br}{{\rm Br}}

\newcommand{\Hom}{{\rm Hom}}

\newcommand{\Spec}{{\rm Spec}}
\newcommand{\kbar}{{\overline{k}}}

\newcommand{\CH}{{\rm CH}}

\newcommand{\Res}{{\rm Res}}

\newcommand{\inv}{{\rm inv}}

\def\Z{{\ZZ}}
\def\Q{{\QQ}}

\begin{document}

\title[]
{Weak approximation of symmetric products and norm varieties}

\author{Sheng Chen}

\address{Sheng Chen \newline School of Mathematical Sciences, \newline  University of Science and Technology of China; \newline 96 Jinzhai Road, 230026 Hefei, China}

\email{chenshen1991@ustc.edu.cn}

\author{ZiYang Zhang}

\address{ZiYang Zhang \newline School of Mathematical Sciences, \newline  University of Science and Technology of China; \newline 96 Jinzhai Road, 230026 Hefei, China}

\email{triangcixjang@mail.ustc.edu.cn}

\date{\today}

\keywords{Brauer--Manin obstruction, weak approximation, symmetric product, norm variety}
\subjclass[2010]{Primary: 11G35, 14G05}

\begin{abstract} Let $k$ be a number field. For a variety $X$ over $k$ that satisfies weak approximation with Brauer--Manin obstruction, we study the same property for smooth projective models of its symmetric products. Based on the same method, we also explore the property of weak approximation with Brauer--Manin obstruction for norm varieties.
\end{abstract}

\maketitle
\section{Introduction}
\subsection{Background} Let $k$ be a number field and $X$ a smooth projective variety over $k$. The set $X(k)$ of $k$-rational points of $X$ can be viewed as a subset of the product $X{({\mathbf A}_k)}=\prod_{v \in \Omega_k} X(k_v)$ via the diagonal embedding, where $\Omega_k$ denotes the set of places of $k$, $k_v$ the completion of $k$ at $v$ and ${\mathbf A}_k$ the ring of ad\`eles of $k$. We endow $\prod_{v \in \Omega_k} X(k_v)$ with the product of the $v$-adic topologies. We say that $X$
satisfies weak approximation if $X(k)$ is dense in $\prod_{v \in \Omega_k} X(k_v)$. The property of weak approximation in general fails for many varieties. Following Manin \cite{Manin}, one may attempt to explain such failures by considering the Brauer--Manin set $X{({\mathbf A}_k)}^{\Br}$, defined as the set of elements of $X({\mathbf A}_k)$ that are orthogonal to the Brauer group $\Br(X)$ of $X$, with respect to the Brauer--Manin pairing. The Brauer--Manin set $X{({\mathbf A}_k)}^{\Br}$ is a closed
subset of $X{({\mathbf A}_k)}$ and the global reciprocity law implies that $X(k)\subset X{({\mathbf A}_k)}^{\Br}$ (see \cite[\S 5.2]{Sko}). We say that $X$
satisfies weak approximation with Brauer--Manin obstruction if $X(k)$ is dense in $X{({\mathbf A}_k)}^{\Br}$.
A conjecture of Colliot-Th\'el\`ene \cite{CT} predicts that $X(k)$ is dense in $X{({\mathbf A}_k)}^{\Br}$ when $X$ is a smooth projective rationally connected variety.

Given a smooth projective variety $X$. Let $Sym^n(X)$ be the $n$-th symmetric product of $X$, i.e., the (singular in general) quotient $X^n/S_n$, where the symmetric group $S_n$ acts by permuting the coordinates. If $X$ is rationally connected, any smooth projective model of $Sym^n(X)$ is also rationally connected. This prompts us to ask the following question.
\begin{question}\label{q1}
Let $X$ be a smooth projective rationally connected variety over $k$. If $X$ satisfies weak approximation with Brauer--Manin obstruction, does any smooth projective model of $Sym^n(X)$ also satisfy this property?
\end{question}

Let $l$ be a prime. Let $a_1, \dots, a_n \in k^{\times}$ with symbol $\alpha :=\{a_1, \dots, a_n\} \in K^{M}_n(k)$. Associated to the symbol $\alpha$ there is a standard norm variety $X_{\alpha,l}$ (see \cite[\S 2]{Sus}), which is constructed by using symmetric product. One may apply this construction to any smooth quasi-projective variety. Briefly, there is a norm variety $N(X,a,n)$ associated to a smooth quasi-projective variety $X$ (see \cite[\S 2]{A}), where $a \in k^{\times}$. Asok \cite[Proposition 2.6]{A} proved that the norm variety $N(X,a,n)$ is rationally connected if $X$ is rationally connected. Then we can also ask the following question.
\begin{question}\label{q2}
Let $X$ be a smooth projective rationally connected variety over $k$. If $X$ satisfies weak approximation with Brauer--Manin obstruction, does any smooth projective model of $N(X,a,n)$ also satisfy this property?
\end{question}
\subsection{Main results} The purpose of this note is to give a positive answer to Question \ref{q1} and \ref{q2} under some assumptions.

\begin{theorem}\label{t1}
Let $X$ be a smooth projective rationally connected variety over $k$. Assume that for any finite field extension $L/k$, the set $X_{L}(L)$
is dense in $X_{L}({\mathbf A}_{L})^{\Br}$. Then any smooth projective model of $Sym^n(X)$ satisfies weak approximation with Brauer--Manin obstruction.
\end{theorem}

In Section 2, we will prove this theorem by using fibration method.

\begin{theorem}\label{t2}
Let $X$ be a smooth projective rationally connected variety over $k$ with $X(k)\ne \emptyset$. Assume that for any finite field extension $L/k$, the set $X_{L}(L)$ is dense in $X_{L}({\mathbf A}_{L})^{\Br}$. Then, assuming the conjecture of Harpaz--Wittenberg (\cite[Conj.9.1]{HW16}), any smooth projective model of $N(X,a,n)$ satisfies weak approximation with Brauer--Manin obstruction.
\end{theorem}

In Section 3, we will give some details of the construction of the norm variety $N(X,a,n)$, and one will see that Theorem \ref{t2} can be proved by the same strategy as Theorem \ref{t1}.

\subsection{Notation and terminology}
Let $K$ be a field of characteristic zero and $\bar{K}$ be a fixed algebraic closure of $K$.
A variety $X$ over $K$ is an integral and separated scheme of finite type over $K$, and we write $K(X)$ for its function field.
Following \cite{Sko96} we call a $K$-variety $X$ is split if it contains a geometrically integral open subscheme. For a variety $X$ over $K$ and a field extension $L/K$, we write $X_{L}=X\times_{K}L$, and set
$$\Pic(X)=H_{\text{\'et}}^1(X, \Bbb G_m), \ \ \Br(X)=H_{\text{\'et}}^2(X, \Bbb G_m).$$
The unramified Brauer group $\Br_{nr}(X/K)$ of $X$ is by definition the Brauer group of any smooth projective model of $X$.

We write $X^n$ for the $n$-fold product of $X$ over $\Spec$ $K$. Let $Sym^n(X)$ be the $n$-th symmetric product of $X$.
If $X$ is smooth and geometrically integral, the symmetric product $Sym^n(X)$ is geometrically integral and normal.

By definition (see \cite[Chapter IV]{Kol96}), a geometrically integral $K$-variety $X$ is rationally connected if and only if two general points of $X_{\bar{K}}$ are connected by a rational curve. Rational connectedness is a birational property for projective varieties. A finite product of rationally connected varieties is also rationally connected.

By definition, a geometrically integral $K$-variety $X$ is rational if there is a birational map
$X_{\bar{K}} \dashrightarrow \Bbb{P}^{n}_{\bar{K}}$. We say $X$ is $K$-rational if there is a birational map
$X \dashrightarrow \Bbb{P}^{n}_{K}$ over $K$. A rational variety is rationally connected.

If $X$ is a smooth projective rationally connected variety over $K$, the group $\Br(X_{\bar{K}})$ is finite and the geometric Picard group $\Pic(X_{\bar{K}})$ is a free abelian group of finite type (see \cite[(6) in p.347]{CS}).

In this article, the letter $k$ always denotes a number field and let $\bar{k}$ be a fixed algebraic closure of $k$.
Let $\Omega_k$ be the set of all places of $k$. For each $v \in \Omega_k$, the completion of $k$ at $v$ is denoted by $k_v$. The ring of ad\`eles of $k$ is denoted by ${\mathbf A}_k$. For a smooth projective variety $X$ over $k$, we set
$$X({\mathbf A}_k)^{\Br}=\{(P_v)_{v \in \Omega_k}\in X({\mathbf A}_k): \sum_{v \in  \Omega_k}\inv_v(\alpha(P_v))=0, \ \ \forall \alpha \in \Br(X)\}.$$
The global reciprocity law in class field theory implies that $X(k)\subset X({\mathbf A}_k)^{\Br} \subset X({\mathbf A}_k)$.

The organization of this note is as follows. In Section 2, we will use fibration method to prove Theorem \ref{t1}. In Section 3, we give details of the construction of the norm varieties and then prove Theorem \ref{t2}.

\section{Weak approximation of symmetric products}
In this section, we will use fibration method to prove Theorem \ref{t1}.
\subsection{Preliminary lemmas} We start by stating several preliminary lemmas. Note that only Lemma \ref{l4} is new.
\begin{lemma}\label{l0}
Let $\rho : V \to W$ be a morphism of smooth varieties over $k$. For any smooth projective model $\overline{W}$ of $W$, there is a smooth projective model $\overline{V}$ of $V$, such that the rational map $\overline{V} \dashrightarrow \overline{W}$ defined by $\rho$ is a morphism.
\end{lemma}

\begin{lemma}\label{l1}
Let $\alpha : X_1 \to X_2$ be a morphism of smooth projective rationally connected varieties over $k$ satisfying $\overline{\alpha(X_1({\mathbf A}_k)^{\Br})}=X_2({\mathbf A}_k)^\Br$ (this is the case if $\alpha$ has a rational section over $k$ and $\alpha^*(\Br(X_2))=\Br(X_1)$). Assume that $X_1(k)$ is dense in $X_1({\mathbf A}_k)^{\Br}$. Then $X_2(k)$ is dense in $X_2({\mathbf A}_k)^{\Br}$.
\end{lemma}
\begin{lemma}(\cite[Proposition 13.3.11]{CS})\label{l2}
Let $k$ be a number field and let $X$ and $Y$ be birationally equivalent smooth, projective, and geometrically integral varieties over $k$. Assume that $\Br(X)/\Br(k)$ is finite. Then $X(k)$ is dense in $X({\mathbf A}_k)^{\Br}$ if and only if $Y(k)$ is dense in $Y({\mathbf A}_k)^{\Br}$.
\end{lemma}

\begin{lemma}\label{l4}
Let $X$ and $Y$ be two smooth projective geometrically integral $k$-varieties. Set $U:= X^n \setminus \Delta$ where $\Delta$ is the union of the diagonals, which is the locus corresponding to $n$ distinct points in $X$. Then we have the following Cartesian diagram.
\begin{equation*}
    \xymatrix{
        U\times Y^n\ar[r]^-{p_0}\ar[d]_{q_1} & U\ar[d]^{q_0}\\
        (U\times Y^n)/S_n\ar[r]^-{p_1} & U/S_n,
    }
\end{equation*}
where the action of the symmetric group $S_n$ on $U$ and $Y^n$ is to permute the coordinates and $S_n$ acts diagonally on $U\times Y^n$, the maps $p_0$ and $p_1$ are natural projections, and the maps $q_0$ and $q_1$ are quotient maps. In particular, the map $p_1$ is a smooth morphism.
\end{lemma}
\begin{proof}
Since the symmetric group $S_n$ acts freely on $U$ and $U\times Y^n$, the maps $q_0$ and $q_1$ are both finite and \'etale by \cite[Expos\'e V, Corollaire 2.3]{SGAI} (in particular $\deg(q_0)=\deg(q_1)$). Thanks to the commutativity of the above diagram, one has the following commutative diagram
\begin{equation*}
    \xymatrix{
        U\times Y^n\ar[r]^--\beta\ar[d]_{q_1} & (U\times Y^n)/S_n\times_{U/S_n}U\ar[d]^\pi\\
        (U\times Y^n)/S_n\ar@{=}[r] & (U\times Y^n)/S_n,
    }
\end{equation*}
where $\beta$ is induced by $p_0$ and $q_1$, and $\pi$ is the first projection. One can check that $\beta$ is surjective, hence the variety $(U\times Y^n)/S_n\times_{U/S_n}U$ is irreducible. Since the variety $(U\times Y^n)/S_n$ is smooth and the map $\pi$ is \'etale, one concludes that $(U\times Y^n)/S_n\times_{U/S_n}U$ is integral and smooth. Since $\deg(q_1)=\deg(q_0)=\deg(\pi)$, one concludes that $\beta$ is a birational map. Note that $\beta$ is also a finite map. Thus $$U\times Y^n\cong (U\times Y^n)/S_n\times_{U/S_n}U$$ by \cite[Corollaire 4.4.9]{EGAIII}. Since $q_0$ is faithfully flat, one concludes that the map $p_1$ is smooth by \cite[Expos\'e II, Corollaire 4.13]{SGAI}.
\end{proof}

\subsection{Proof of Theorem \ref{t1}}
Consider the following natural morphisms
\begin{equation}\label{map}
  X^n/S_n \overset{g}\longleftarrow (X^n\times \Bbb{A}_k^n)/S_n \overset{f}\longrightarrow \Bbb{A}_k^n/S_n,
\end{equation}
where the action of $S_n$ on $\Bbb{A}_k^n$ is just permuting the coordinates, and $S_n$ acts diagonally on $X^n\times \Bbb{A}_k^n$. In particular, the quotient $\Bbb{A}_k^n/S_n$ is isomorphic to $\Bbb{A}_k^n$, as the ring $k[x_1, \dots, x_n]^{S_n}$ of symmetric polynomials coincides with the polynomial ring in the elementary symmetric polynomials.

We denote by $E_{ij}$ the closed subset $V(x_i-x_j)\subset \Bbb{A}_k^n$ for $i,j \in \{1,\dots, n\}$ and set $\Delta=\bigcup_{i\ne j}E_{ij}$. Set $U:=(\Bbb{A}_k^n \setminus \Delta)/S_n$, which is an open subset of $\Bbb{A}_k^n/S_n$. The projection $$\pi: \Bbb{A}_k^n \setminus \Delta \to U$$ is a Galois \'etale covering (cf. \cite[Section 1.2]{Wit23}) and then we define a Hilbert subset $H$ of $\Bbb{A}_k^n/S_n$ to be the set of points of $U$ above which the fiber of $\Bbb{A}_k^n \setminus \Delta$ is irreducible. By \cite[Corollary 1.3.2]{Wit23}, we have $H\cap U(k)\ne \emptyset$. For any $y\in H\cap U(k)$, the fiber $\pi^{-1}(y)$ is an $S_n$-torsor over $k$, and its affine ring is a field $$K=k[x_1,\dots, x_n]/I=k[\overline{x}_1,\dots, \overline{x}_n],$$ where $I$ is a maximal ideal of $k[x_1,\dots, x_n]$, and $\overline{x}_i$ is the image of $x_i$ in $k[x_1,\dots, x_n]/I$ for $i \in \{1, \dots, n\}$. In particular $K$ is a Galois extension of $k$ with Galois group $S_n$, and $\{\overline{x}_i\}_{1\le i \le n}$ are all roots of an irreducible polynomial of degree $n$ over $k$.

Let $S_{n-1}\subset S_n$ be the subgroup of permutations that fix the first factor. The Galois closure of $K^{S_{n-1}}(=k(\overline{x}_1))$ over $k$ is $K$.
This implies that the fiber
\begin{equation}\label{equres}
f^{-1}(y)(=(X^{n}\times \pi^{-1}(y))/S_n)\cong \Res_{k(\overline{x}_1)/k}(X_{k(\overline{x}_1)}).
\end{equation}

Since the groups $\Br(\Res_{k(\overline{x}_1)/k}(X_{k(\overline{x}_1)}))/\Br(k)$ and $\Br(X_{k(\overline{x}_1)})/\Br(k(\overline{x}_1))$ are both finite (cf. \cite[(6) in p.347 ]{CS}), the sets $\Res_{k(\overline{x}_1)/k}(X_{k(\overline{x}_1)})(\mathbf A_{k})^{\Br}$ and $X_{k(\overline{x}_1)}(\mathbf A_{k(\overline{x}_1)})^{\Br}$ are open in $\Res_{k(\overline{x}_1)/k}(X_{k(\overline{x}_1)})(\mathbf A_{k})$ and $X_{k(\overline{x}_1)}(\mathbf A_{k(\overline{x}_1)})$, respectively, by \cite[Proposition 13.3.1 (iv)]{CS}. Then we can see that the fiber $f^{-1}(y)$ satisfies weak approximation with Brauer--Manin obstruction by \cite[Corollary 3.2]{CL}, \cite[Proposition 5.14]{Brian} and our assumption on $X$. Note that by Lemma \ref{l4}, the generic fiber of $f$ is rationally connected, since $X^n$ is rationally connected. Thus we obtain that any smooth projective model of $(X^n\times \Bbb{A}_k^n)/S_n$ satisfies weak approximation with Brauer--Manin obstruction by Lemma \ref{l3}, \cite[Theorem 3.6]{HW22} and Lemma \ref{l2}. Since the generic fiber of $g$ is $k(X^n/S_n)$-rational (see \cite[Lemma 2.4]{K}), we have
$$\Br_{nr}(k(X^n/S_n)/k)=\Br_{nr}(k((X^n\times \Bbb{A}_k^n)/S_n)/k)$$ by \cite[Proposition 6.2.9]{CS}, and in particular $g$ has rational sections over $k$. Thus any smooth projective model of $X^n/S_n$ satisfies weak approximation with Brauer--Manin obstruction by Lemma \ref{l0}, \ref{l1} and \ref{l2}.
\qed

\begin{lemma}[suggested by Olivier Wittenberg]\label{l3}

Taking a smooth projective model $Z$ of $(X^n\times \Bbb{A}_k^n)/S_n$ such that the rational map $Z \dashrightarrow {\Bbb{P}_k^n}$ defined by ${f}$ is a morphism and we denote by $\widehat{f}$ this morphism. Then the fiber of $\widehat{f}$ above
any codimension 1 point of ${\Bbb{P}_k^n}$ is split.

\end{lemma}
\begin{proof}
For any codimension 1 point $P$ of ${\Bbb{P}_k^n}$, we write $R_P$ for its local ring, which is a discrete valuation ring. Consider the following morphism
$$j: Z\times_{{\Bbb{P}_k^n}} \Spec(R_P)\to \Spec(R_P),$$
where $j$ is the second projection and $j$ is a proper surjective map. In particular $j$ is also flat by \cite[Lemma 10.3.A]{Hartshorne}.
Take the codimension 1 point $P'$ of ${\Bbb{P}_{k(X)}^n}: ={\Bbb{P}_{k}^n}\times_k k(X)$ above $P$. Note that $k$ is algebraically closed in $k(X)$. Such point exists and is unique by \cite[Chap II, Exercise 3.20 (f)]{Hartshorne}, \cite[Proposition 4.3.2]{EGAIV2}.
Write $R_{P'}$ for the local ring of $P'$ in ${\Bbb{P}_{k(X)}^n}$ and $K'$ the function field of $R_{P'}$. Since $X(k(X))\ne \emptyset$, the map $f$ (see \ref{map}) has a section over $k(X)$. Thus the generic fiber of $j$ has a $K'$-point. By \cite[Proposition 4.3.2, Corollaire 4.3.6]{EGAIV2}, one can see that the local extension of discrete valuation ring $R_P\subset R_{P'}$ satisfies the conditions of \cite[Lemma 1.1 (b)]{Sko96}. Then by this lemma the special fiber of $j$ is split.
\end{proof}

\begin{remark}
Given a smooth projective variety $Y$ over $k$, one has the following complex (see \cite[\S 15.1]{CS}, \cite[\S 1.1]{Wit12} for a detailed definition)
\begin{equation}\label{complex}
\varprojlim_{n}\CH_{0}(Y)/n\to\prod_{v \in \Omega_{k}}\varprojlim_{n}\CH'_{0}(Y_{k_v})/n\to \Hom(\Br(X),\Q/\Z),
\end{equation}
where $\CH_0(Y)$ is the Chow group of $0$-cycles, the modified local Chow group $\CH'_0(Y_{k_v})$ is defined to be the usual Chow group $\CH_0(Y_{k_v})$ if $v$ is a non-archimedean place and otherwise $\frac{\CH_0(Y_{k_v})}{N_{\kbar{_v}/k_v}(\CH_0(Y_{\kbar_v}))}.$

By combining Theorem \ref{t1} and a result of Liang \cite[Theorem 3.2.1]{Li}, we can obtain that the complex \ref{complex} is exact for any smooth projective model of $Sym^n(X)$.
\end{remark}

\section{Weak approximation of norm varieties}
In this section, we first outline the construction of the norm varieties (cf. \cite[\S 2]{A}) and then prove Theorem \ref{t2}.

\subsection{Constructing norm varieties} Let $X$ be a smooth quasi-projective geometrically integral $k$-variety. Write $X^n$ for the $n$-fold product. The symmetric group $S_n$ acts on $X^n$ by permuting the factors. In particular, the symmetric group $S_n$ acts freely on the open subscheme $U:= X^n \setminus \Delta$ where $\Delta$ is the union of the diagonals, which is the locus corresponding to $n$ distinct points in $X$. A geometric quotient $C_n{(X)}:=U/S_n$ exists as a smooth open subscheme of $Sym^n(X)$.

There is a natural morphism $p: X \times Sym^{n-1}(X) \to Sym^n(X)$; this morphism is finite surjective of degree $n$.
The preimage of $C_n(X)$ under $p$ gives rise to a morphism
$$X \times Sym^{n-1}(X) \supset p^{-1}(C_n{(X)}) \longrightarrow C_n(X),$$
which is a finite \'etale morphism of degree $n$. In particular, $p^{-1}(C_n{(X)})$ is a smooth scheme.

Set \begin{equation}\label{equ1}
    \mathcal{A}_X:= p_{\ast}(\mathscr{O}_{ X \times Sym^{n-1}(X) })|_{C_n(X)},
\end{equation}
which is a locally free sheaf of $\mathscr{O}_{C_n(X)}$-algebras of rank $n$.
We let
\begin{equation}\label{equ2}
\psi: \mathbb{V}(\mathcal{A}_X) \to C_n(X)
\end{equation}
be the associated geometric vector bundle (given by taking the spectrum of the symmetric algebra of the dual of $\mathcal{A}_X$). Since $\mathcal{A}_X$ is a locally free algebra, there is a well-defined norm morphism
\begin{equation}\label{equ3}
N:\mathcal{A}_X \to \mathscr{O}_{C_n(X)},
\end{equation}
which can be identified as a section of $Sym_n(\mathcal{A}_{X}^{\vee})$.
Take an element $a \in k^\times$. Let $N(X,a,n)\subset \mathbb{V}(\mathcal{A}_X)$ be the closed
subscheme defined by the equation $N-a = 0$. Denoted by
\begin{equation}\label{equ4}
\Psi: N(X,a,n) \to C_n(X)
\end{equation}
the composition of the inclusion $N(X,a,n)\subset \mathbb{V}(\mathcal{A}_X)$ with $\psi$. By \cite[Lemma 2.1]{Sus}, the variety $N(X,a,n)$ is smooth over $C_n{(X)}$ (smooth over $k$) and geometrically irreducible. In particular, by \cite[Proposition 2.6]{A} the variety $N(X,a,n)$ is rationally connected if $X$ is rationally connected. One calls $N(X,a,n)$ a norm variety (or Rost variety).

\begin{lemma}\label{p1}
Assuming the conjecture of Harpaz--Wittenberg (\cite[Conj.9.1]{HW16}). Then any smooth projective model of $N(\Bbb{A}_{k}^1,a,n)$ satisfies weak approximation with Brauer--Manin obstruction.
\end{lemma}
\begin{proof}
By \ref{equ4}, we have the following morphism
$$\varphi: N(\Bbb{A}_{k}^1,a,n) \to C_n{(\Bbb{A}_{k}^1)}\subset \Bbb{A}_{k}^n.$$
For any rational point $c$ of $C_n{(\Bbb{A}_{k}^1)}$, the fiber $\varphi^{-1}(c)$ is isomorphic to the norm hypersurface $$\{N_{A/k}(z)-a=0\} \subset \Bbb{A}_{k}^{n},$$ where $A$ is a finite \'etale $k$-algebra with $[A:k]=n$ and $z=(z_1,\dots, z_n)$.
By \cite[Theorem 6.3.1]{Sko},  any smooth projective model of $\varphi^{-1}(c)$ satisfies weak approximation with Brauer--Manin obstruction. Note that the generic fiber of $\varphi$ is rationally connected. Now this lemma follows from Lemma \ref{l0}, \ref{l2} and \cite[Corollary 9.25]{HW16}.
\end{proof}

\subsection{Proof of Theorem \ref{t2}}
Let $$p'_0: X \times Sym^{n-1}(X) \to Sym^n(X); \ \ \ p'_1:  X \times \Bbb{A}_{k}^1 \times Sym^{n-1}(X \times \Bbb{A}_{k}^1) \to Sym^n(X \times \Bbb{A}_{k}^1)$$
$$  r'_0: X \times \Bbb{A}_{k}^1 \times Sym^{n-1}(X \times \Bbb{A}_{k}^1)\to X \times Sym^{n-1}(X); \ \ \
 r'_1:  Sym^{n}(X \times \Bbb{A}_{k}^1) \to Sym^n(X)$$ be the natural morphisms.
Take the open subscheme $U$ of $X^n$ which is the locus corresponding to $n$ distinct points in $X$. Then we have the following commutative diagram

\begin{equation*}
    \xymatrix{
        p_1^{-1}((U\times (\Bbb{A}_{k}^1)^n)/S_n)\ar[r]^-{r_0}\ar[d]_{p_1} & p_0^{-1}(U/S_n)\ar[d]^{p_0}\\
        (U\times (\Bbb{A}_{k}^1)^n)/S_n\ar[r]^--{r_1} & U/S_n,
    }
\end{equation*}
where $p_i$ (resp. $r_i$) is the restriction of $p'_i$ (resp. $r_i$) to the corresponding open subset for $i \in \{0, 1\}$.
Arguing in the same way as in the proof of Lemma \ref{l4}, one can show that the diagram is Cartesian. By Lemma \ref{l4}, the map $r_1$ is smooth. Now we claim that the variety
\begin{equation}\label{eq7}
N(X,a,n)\times_{U/S_n}(U\times (\Bbb{A}_{k}^1)^n)/S_n
\end{equation}
is an open subscheme of $N(X\times \Bbb{A}_{k}^1,a,n)$. Indeed, by \cite[Chapter III, Proposition 9.3]{Hartshorne}, we have
\begin{equation}\label{equ5}
r^{*}_1(\mathcal{A}_X)= \mathcal{A}_{{X\times \Bbb{A}_{k}^1}}|_{(U\times (\Bbb{A}_{k}^1)^n)/S_n}
\end{equation}
(see \ref{equ1}).
Since $\mathcal{A}_X$ is a locally free sheaf, we have
$r^{*}_1({\mathcal{A}_X^{\vee}})=\mathcal{A}_{{X\times \Bbb{A}_{k}^1}}^{\vee}|_{((U\times (\Bbb{A}_{k}^1)^n)/S_n)}$.
In particular we have
\begin{equation}\label{eq6}
 r^{*}_1Sym({\mathcal{A}_X^{\vee}})=Sym(r^{*}_1({\mathcal{A}_X^{\vee}}))=Sym(\mathcal{A}_{{X\times \Bbb{A}_{k}^1}}^{\vee}|_{(U\times (\Bbb{A}_{k}^1)^n)/S_n})
\end{equation}
by \cite[\S 1.7 (1.7.5)]{EGAII}. Thus we have $$\mathbb{V}(\mathcal{A}_X)\times_{U/S_n}(U\times (\Bbb{A}_{k}^1)^n)/S_n= \mathbb{V}(\mathcal{A}_{{X\times \Bbb{A}_{k}^1}}^{\vee}|_{(U\times (\Bbb{A}_{k}^1)^n)/S_n})\subset \mathbb{V}(\mathcal{A}_{{X\times \Bbb{A}_{k}^1}}^{\vee})$$
by \cite[Proposition 1.7.11]{EGAII}.
By the construction of norm varieties, we conclude that the variety $$N(X,a,n)\times_{U/S_n}(U\times (\Bbb{A}_{k}^1)^n)/S_n$$
is an open subscheme of $N(X\times \Bbb{A}_{k}^1,a,n)$.

Let
$$\chi : N(X,a,n)\times_{U/S_n}(U\times (\Bbb{A}_{k}^1)^n)/S_n\to N(X,a,n)$$
be the first projection. By \cite[Lemma 2.4]{K}, the varieties $(U\times (\Bbb{A}_{k}^1)^n)/S_n$ and ${U/S_n}\times \Bbb{A}_{k}^{n}$ are birationally equivalent. This implies that the map $\chi$ has rational sections over $k$ and $$\Br_{nr}(k(N(X,a,n))/k)=\Br_{nr}(k(N(X\times \Bbb{A}_{k}^1,a,n))/k)$$ by \cite[Proposition 6.2.9]{CS}. Thus, to prove our theorem, it's enough to show that any smooth projective model of $N(X\times \Bbb{A}_{k}^1,a,n))$ satisfies weak approximation with Brauer--Manin obstruction by Lemma \ref{l0}, \ref{l1} and \ref{l2}.

Now we take the open subscheme $V$ of $(\Bbb{A}_{k}^1)^n$ which is the locus corresponding to $n$ distinct points in $\Bbb{A}_{k}^1$. By a similar argument as above, we have the following morphism
$$ \Pi : N(\Bbb{A}_{k}^1,a,n)\times_{V/S_n}(V\times X^n)/S_n\to N(\Bbb{A}_{k}^1,a,n),$$
where $N(\Bbb{A}_{k}^1,a,n)\times_{V/S_n}(V\times X^n)/S_n$ is an open subscheme of $N(X\times \Bbb{A}_{k}^1,a,n)$. Note that the natural smooth (see Lemma \ref{l4}) morphism $(V\times X^n)/S_n \to {V/S_n}$ has sections over $k$, as $X(k)\ne \emptyset$. Thus the morphism $\Pi$ has sections over $k$, and the generic fiber of $\Pi$ has a smooth $k(N(\Bbb{A}_{k}^1,a,n))$-point. By Lemma \ref{l4}, we know that the morphism $(V\times X^n)/S_n \to {V/S_n}$ has rationally connected generic fiber. Thus the generic fiber of $\Pi$ is also rationally connected.

Recall that $V \to V/S_n$ is a Galois \'etale covering. We define a Hilbert subset $H$ of $N(\Bbb{A}_{k}^1,a,n)$ to be the set of points of $N(\Bbb{A}_{k}^1,a,n)$ above which the fiber of $\pi$ is irreducible, where $$\pi: V \times_{V/S_n}{N(\Bbb{A}_{k}^1,a,n)} \to N(\Bbb{A}_{k}^1,a,n)$$ is the second projection. Then for any $y\in H\cap N(\Bbb{A}_{k}^1,a,n)(k)$, the fiber $\Pi^{-1}(y)$ is a Weil restriction of $X$ from an \'etale algebra (see \ref{equres}). In particular, the fiber $\Pi^{-1}(y)$ satisfies weak approximation with Brauer--Manin obstruction (see proof of Theorem \ref{t1}). Now we can see that any smooth projective model of $N(X\times \Bbb{A}_{k}^1,a,n)$ satisfies weak approximation with Brauer--Manin obstruction by Lemma \ref{p1} and the proof of \cite[Th\'eor\`eme 3]{Hara07}. The proof is complete.
\qed

\medskip

\bf{Acknowledgments}	
\it{ The authors would like to thank Yang Cao for suggesting to us the topic of this article and generously sharing his ideas with us. Without his help it wouldn't be possible to see this article coming out. Special thanks to Olivier Wittenberg for his valuable comments, especially for his suggestion on Lemma 2.5.
We would also like to thank Fei Xu, Yongqi Liang and Guang Hu for their useful suggestions.
The first author is partially supported by National Natural Science Foundation of China No.12071448.}

%%%%%%%%%%%%%%%%%%%%%%%%%%%%%%%%%%%%%%%%%%%%%%%%%%%%%%%%%%%%%%%%%%%%%%%%%%%%%%%%%%%%%

\end{document}